\documentclass{amsart}

\usepackage{bm}
\usepackage{graphicx}
\usepackage{enumerate}
\usepackage{amsthm}
\usepackage{amsmath}
\usepackage{amssymb}

\newcommand{\ti}{\tilde}

\newcommand{\R}{\mathbb{R}}

\newcommand{\diam}{\text{diam}}

\renewcommand{\phi}{\varphi}

\newcommand{\mL}{\mathcal{L}}
\newcommand{\mM}{\mathcal{M}}
\newcommand{\mD}{\mathcal{D}}
\newcommand{\mP}{\mathcal{P}}
\newcommand{\mF}{\mathcal{F}}
\newcommand{\mG}{\mathcal{G}}

\newtheorem{theorem}{Theorem}
\newtheorem{proposition}[theorem]{Proposition}

\newtheorem*{lemma*}{Lemma}
\newtheorem*{thmA}{Theorem A}

\theoremstyle{definition}
\newtheorem{definition}[theorem]{Definition}

\theoremstyle{remark}
\newtheorem{example}[theorem]{Example}
\newtheorem{remark}{Remark}

\author{James Vargo}

\title{Reconstructing curves from lengths of projections onto lines}

\begin{document}

\begin{abstract}
In this paper, we address the problem of reconstructing a curve from the lengths of its projections onto lines.  
We first note that the curve itself is not uniquely determined from 
these measurements.  However, we find that a curve determines a measure on 
projective space which, as a function on Borel subsets of
projective space, returns the 
length of curve parallel to elements of the set.  We show that the projected length data can be expressed as the 
cosine transform of this measure on projective space.  The cosine transform is a well studied
integral transform on the sphere which is known to be injective.  We conclude that the measured length data uniquely determines the 
associated measure on projective space.  We then characterize the class of curves that
produce a common measure by starting with the case of piecewise linear curves and then passing to limits to obtain results for more general curves.
\end{abstract}

\maketitle

\section{Introduction}
\subsection{The problem}
Let $\alpha:[0,1]\rightarrow \mathbb{R}^d$ be an absolutely continuous curve, and let $\xi\in S^{d-1}$ be a unit vector.
For any line parallel to $\xi$, the length of the projection of
$\alpha$ onto the line is given by the integral:
\begin{align}\label{M}
M_\alpha(\xi)=\int_0^1 |\xi\cdot \alpha'(t)|\, dt.
 \end{align}

In this paper, we shall investigate what properties of $\alpha$
can be reconstructed from the data $M_\alpha:S^{d-1}\rightarrow \R$.\\

First, we note that the measurements can at best reconstruct properties of the velocity curve $\alpha'(t)$.  Indeed, $M_\alpha$ remains unchanged if the curve is translated in $\mathbb{R}^n$.  
Therefore, letting $\beta=\alpha'$, 
we could rephrase the problem to reconstruct $L^1$ curves $\beta$ from the integrals $M_\beta(u)=\int_a^b |u\cdot \beta(t)|\, dt.$ In light of the tomographic motivation presented in the next 
section, we prefer to consider the problem in terms of curves.
A variation of this problem not considered in this paper (but perhaps more fruitful for tomography) would be to add a weight function $f:\mathbb{R}^n\rightarrow \mathbb{R}$ to the measured data in the form
\begin{align}\label{weight}
  M_\alpha(\xi)=\int_0^1f(\alpha(t))|\xi\cdot \alpha'(t)|\, dt.
 \end{align}

\subsection{Motivation from tomography}
Consider a medium with unknown optical properties. We set up an apparatus that does two things. First, it imposes a uniform magnetic field over the medium. Second, it 
sends photons out from a point A. These photons pass through the medium, possibly scattering (bouncing around) many times and possibly tracing curved paths due to a (possibly) non-constant index of refraction.
Some of the photons reach point B and their respective net change in polarity is measured. The change in polarity of a given photon depends on two factors: the magnetic field and the path taken from point A to point B.  
The magnetic field is known because it is imposed by the experimenter. Therefore, the measured change in polarity carries information about the path taken. Therefore, it must carry information about the optical
properties of the medium itself. By some averaging process, the optical properties of the medium could perhaps be deduced from measurements of many photons.\\

If the magnetic field is given by $\xi$, then the measured data, the net change in polarity of the photon, is given by formula \eqref{M}. 
If the magnetic field has fixed direction but variable magnitude, then the measured data takes the form \eqref{weight}. So the problem addressed in this paper seeks to address a part of this tomography problem.
However, we acknowledge that measuring $M_\alpha(\xi)$ for all $\xi$ is not a direct model of the experiment since it represents an infinite number of measurements taken on a single photon.

The path of a photon in a given medium can be quite chaotic in the sense that it could scatter many times like a diffusive particle. However, its path is necessarily Lipschitz, hence absolutely continuous. In fact, in the 
context of this problem, there is no essential difference between Lipschitz and absolutely continuous (see Theorem \ref{constspeedT}).

\subsection{The data determines a measure on projective space}
We let $P^{d-1}$ denote real projective space of dimension $d-1$, and let $\pi:S^{d-1}\rightarrow P^{d-1}$ denote the projection that identifies antipodal points.  
Given a function $f$ on projective space, we let
$\hat{f}:\R^d\rightarrow \R$ denote the degree $1$ homogeneous extension of $f\circ \pi$:
\begin{align*}
 \hat{f}(r\xi)=|r|(f\circ\pi)(\xi),
\end{align*}
for $r\in \R$ and $\xi\in S^{d-1}$.  According the the Riesz Representation theorem, the curve $\alpha$ uniquely determines a measure 
$\mu_\alpha$ on $P^{d-1}$ through the positive linear functional on $C(P^{d-1})$:
\begin{align}\label{functional}
I_\alpha(f) = \int_0^1 \hat{f}(\alpha'(t))\, dt
 \end{align}

For each $\xi\in S^{d-1}$, we define $f_\xi:P^{d-1}\rightarrow \R$ by
\begin{align*}
  f_\xi(x) = |\xi\cdot \pi^{-1}(x)|.
 \end{align*}

Then we find that
\begin{align}\label{Mcosine}
 M_\alpha(\xi) & = \int_0^1|\xi\cdot \alpha'(t)|\, dt \notag\\
& = I_\alpha(f_\xi)  \\
& =\int f_\xi\, d\mu_\alpha. \notag
\end{align}
 Given any finite Borel measure $\mu$ on the sphere, the cosine transform of $\mu$ is defined to be
\begin{align*}
\mathcal{C}[\mu](\xi) =  \int_{S^{d-1}} |\xi \cdot x|\, \mu(x).
\end{align*}
The cosine transform has been studied by many authors.  In particular it is known that the linear span of the functions $f_\xi$ 
is uniformly dense in $C(P^{d-1})$, and that the transform is invertible \cite{Bolk}.  For measures with smooth
densities, explicit inversion formulas have been calculated using spherical harmonics (\cite{Sch}).
More recently, inversion formulas were found for $L^p$ densities (\cite{Rubin02}). In \cite{Kid, Hoff, Gard, Louis} various 
numerical reconstructions of the measure are derived.\\

By formula \eqref{Mcosine}, we conclude that $M_\alpha$ is the cosine transform
of the measure $\mu_\alpha\circ \pi^{-1}$ on $S^{d-1}$.  By injectivity, the measures $\mu_\alpha\circ \pi^{-1}$ on $S^{d-1}$ and 
$\mu_\alpha$ on $P^{d-1}$ are uniquely determined from $M_\alpha$.  If the measure $\mu$ has the form
\[\mu = f\, d\sigma,\]
where $\sigma$ is the standard surface measure of the sphere, then (\cite{Louis}) there exists $k>0$ such that 
\[
 k^{-1} \|f \|_{H^s} \leq \|\mathcal{C} f\|_{H^{s+\frac{d+2}{2}}}\leq k\|f\|_{H^s}.
\]
The Sobolev spaces are defined using the usual norms with respect to the spherical harmonics.

\subsection{The correspondence between curves and their associated measures on projective space}
We let $AC$ denote the space of absolutely continuous curves $\alpha:[0,1]\rightarrow \R^d$ modulo translation.  We give it topology through the $L^1$ norm of the velocity function:
\begin{align}\label{ACnorm}
  \|\alpha\|_{AC}=\|\alpha'\|_{L^1[0,1]}.
\end{align}
Piecewise linear curves (also referred to as broken lines) are a subclass which are dense in $AC$ and which can be defined as concatenations of linear 
segments.\footnote{Density follows from the density of step functions in $L^1[0,1]$.}
The decomposition of a broken line into linear segments
is not unique since a linear segment can itself be decomposed into smaller subsegments.
\begin{definition}
 Two broken lines are \emph{rearrangements} of one another if they can each be decomposed in such a way that their respective segments are translates of one another.
\end{definition}
For example, in a parallelogram $ABCD$, the broken
lines $ABC$, $ADC$ and $CBA$ are all rearrangements of each other.  Moreover, we could further decompose the segments $AB$ and $BC$ into subsegments and rearrange them to form yet more rearrangements.  
It is clear that two broken lines are rearrangements of one another if and only if they produce the same measure on $P^{d-1}$ which is necessarily a discrete measure.\\

\begin{proposition}
 The broken line $\alpha = A_0A_1\dots A_n$ produces the discrete measure with weights $c_k=|A_{k+1}-A_k|$ supported at the points $p_k=\pi\left(\frac{A_{k+1}-A_k}{|A_{k+1}-A_k|}\right)$:
\begin{align*}
 \mu_\alpha = \sum_{k=0}^{n-1} c_k\delta_{p_k}.
\end{align*}
\end{proposition}

The following proposition is a technical necessity and is proved in the appendix.

\begin{proposition}
 For each curve $\alpha\in AC$, there is a reparametrization $\ti{\alpha}$ which is constant speed and 
 which produces the same measure $\mu_\alpha$.
\end{proposition}

The following theorem is proved by constructing a right inverse.
\begin{theorem}
 The correspondence $\alpha\mapsto \mu_\alpha$ is surjective on the set of positive, finite Borel measures on $P^{d-1}$.
\end{theorem}

\section{Continuity Results}
We have three spaces, each equipped with a topology:

1. The space of absolutely continuous curves modulo translation with the norm $\|\cdot \|_{AC}$ defined by \eqref{ACnorm}:
\begin{align*} 
\|\alpha\|_{AC}=\|\alpha'\|_{L^1[0,1]}.
 \end{align*}

2. The space $\mM(P^{d-1})$ of finite signed Borel measures on $P^{d-1}$.  We also define $\mM^+(P^{d-1})\subset \mM(P^{d-1})$ to be the subset of positive, finite Borel measures.  
We take the standard weak topology in which $\mu_k\rightarrow \mu$ if
 $\int f\, \mu_k\rightarrow \int f\, \mu$ for all continuous functions $f$.  On compact spaces, this topology is metrizable through Lipschitz functions.  Let $d_E$ denote the Euclidean
 distance in $\R^d$, and let $d_S$ denote its restriction to points in the sphere $S^{d-1}$.  We a metric $d_P$ on projective space $P^{d-1}$ by:
 \begin{align*} 
 d_P(x,y)= d_S(\pi^{-1}(x),\pi^{-1}(y)).
   \end{align*}
 Through this metric, we define $\mL(P^{d-1})$, the space of Lipschitz functions on $P^{d-1}$, and endow it with the norm $\|\cdot\|_{\mL}$ defined by
 \begin{align*} 
 \| f\|_\mL = \|f\|_{\infty}+\sup_{x\neq y} \frac{|f(y)-f(x)|}{d_P(x,y)}.
   \end{align*}

The weak topology on $\mM(P^{d-1})$ is metrizable through the norm $\|\cdot\|_w$ defined by
 \begin{align*}
 \|\mu\|_w = \sup_{\|f\|_\mathcal{L}\leq 1} \int f\, d\mu,
 \end{align*}
 where the supremum is taken over all Lipschitz functions with $\mL$ norm less than or equal to $1$.\\
 
3. The space $C(P^{d-1})$ of continuous functions on $P^{d-1}$ with the uniform norm $\|\cdot \|_{\infty}$.  We identify $C(P^{d-1})$ with the space of even
continuous functions on the sphere, $C_E(S^{d-1})$.\\

With respect to these norms we have the following continuity results.
\begin{proposition}
For all absolutely continuous curves $\alpha, \beta \in AC$, and all finite signed Borel measures $\mu,\nu \in \mM(P^{d-1})$,
\begin{align}
  \|\mu_{\alpha}-\mu_{\beta}\|_w\  \leq  \|\alpha-\beta\|_{AC}; \label{muCont}\\
  \|M_\mu-M_\nu\|_\infty \leq  2\|\mu-\nu\|_w.
 \end{align}
\end{proposition}
To prove the first inequality we need the following lemma.
\begin{lemma*} If $f:P^{d-1}\rightarrow \mathbb{R}$ is Lipschitz, then its homogeneous extension $\hat{f}:\mathbb{R}^d\rightarrow \mathbb{R}$
is Lipschitz with constant $\|f\|_\mL$.
\end{lemma*}
\begin{proof}  Let $f:P^{d-1}\rightarrow \mathbb{R}$ be 
Lipschitz with constant $B_1=\sup_{x\neq y}\frac{|f(y)-f(x)|}{d_P(x,y)}$.  Let $\|f\|_\infty=B_2$.  We will show that $\hat{f}$ is Lipschitz with 
constant $\|f\|_\mL = B_1+B_2$.\\

Let $x,y$ be arbitrary vectors and without loss of generality, assume $|x|\leq |y|$.  If $x=0$, then 
\begin{align*} 
  |\hat{f}(x)-\hat{f}(y)|=|\hat{f}(y)|\leq B_2|y|\leq (B_1+B_2)|x-y|.
 \end{align*}

Next assume that $|x|=1$, and $|y|=r\geq 1$.  Let $y=rz$.
\[\begin{array}{rl}|\hat{f}(x)-\hat{f}(y)|& \leq |\hat{f}(x)-\hat{f}(z)|+|\hat{f}(z)-\hat{f}(y)| \\
   & \leq |f(x)-f(z)|+|f(z)-rf(z)|\\
   & \leq B_1|x-z|+|f(z)|(r-1)\\
   & \leq B_1|x-z|+B_2|y-z|.
  \end{array}\]
We note that $y$ lies outside the unit sphere, $x$ is a point on the unit sphere,
and $z$ is the closest point on the sphere to $y$.  It follows that
$|z-y|\leq|x-y|$, and $|x-z|\leq |x-y|$ (angle $xzy$ is obtuse).\\

Finally, if $|x|=s\neq 0$, then
\begin{align*}
|\hat{f}(x)-\hat{f}(y)| & = s|\hat{f}(x/s)-\hat{f}(y/s)| \\
 & \leq (B_1+B_2)s|x/s-y/s| \\
 & =(B_1+B_2)|x-y|.
\end{align*}
\end{proof}

\begin{proof}
Using the lemma, we find
\[ \begin{array}{rl} \left|\int_0^1 \hat{f}(\alpha'(t))-\hat{f}(\beta'(t))\, dt\right| & \leq \int_0^1|\hat{f}(\alpha'(t))-\hat{f}(\beta'(t))|\, dt \\
    &  \leq \|f\|_\mL \int_0^1|\alpha'(t)-\beta'(t)|\, dt \\
    &  =\|f\|_\mL \|\alpha-\beta\|_{AC}\end{array}.\]
Taking the supremum over all Lipschitz $f$ with $\|f\|_\mL \leq 1$, we obtain inequality \eqref{muCont}.\\

For all $\xi\in S^{d-1}$, we note that $f_\xi(x)=|\xi\cdot \pi^{-1}(x)|$ is a Lipschitz function on $P^{d-1}$ with $\|f_\xi\|_{\mL}\leq 2$.  Therefore
\begin{align*}
 \|M_\mu - M_\nu\|_\infty &= \sup_{\xi\in S^{d-1}} \left|\int f_\xi\, d(\mu-\nu)\right| \\
 & \leq 2\|\mu-\nu\|_w
\end{align*}
\end{proof}

\section{Operations on curves}
In this section, we consider operations on curve that preserve the induced measure on projective space.

\begin{theorem}\label{constspeedT}
Given $\alpha\in AC$, there exists a right-continuous, strictly monotonic function 
$\ti{\phi}:[0,1]\rightarrow [0,1]$ with the following properties:
\begin{enumerate}[i.]
\item  For almost all $s$, the composition $\alpha\circ\ti{\phi}$ satisfies
\begin{align*}
 |(\alpha\circ\ti{\phi})'(s)|=\|\alpha\|_{AC}\hspace{.3cm} \text{and}\\
(\alpha\circ\ti{\phi})'(s)=\alpha'(\ti{\phi}(s))\ti{\phi}'(s).
\end{align*}
\item  If $F:[0,1]\rightarrow \R^k$ is absolutely continuous and
\begin{align*}
 \{t\, |\, F'(t) \neq 0\} \subset \{t\, |\, \alpha'(t)\neq 0\},
 \end{align*}
 then $F\circ\ti{\phi}$ is absolutely continuous.  In particular $\alpha\circ\ti{\phi}$ is absolutely continuous.
\item  $\mu_{\alpha}=\mu_{\alpha\circ \ti{\phi}}$.  In particular $\|\alpha\|_{AC}=\|\alpha\circ\ti{\phi}\|_{AC}$.
\end{enumerate}
\end{theorem}
This (probably well-known fact) is proved in detail in the appendix.\\

Given a curve $\alpha\in AC$, define $\alpha^{-}(t)=\alpha(1-t)$.  
\begin{definition}
 Given two curves $\alpha$ and $\beta$ in $AC$, let $p=\frac{\|\alpha\|_{AC}}{\|\alpha\|_{AC}+\|\beta\|_{AC}}$, and define the {\em concatenation} $\alpha\oplus\beta$ by
\[\alpha\oplus\beta(t)=\left\{\begin{array}{lr}
                               \alpha(\frac{t}{p}), & 0\leq t\leq p;\\
                               \beta(\frac{t-p}{1-p})+\alpha(1)-\beta(0),  & p\leq t \leq 1.\\
                              \end{array}\right.\]
\end{definition}
Concatenation is associative and satisfies $\alpha\oplus0=0\oplus\alpha=\alpha$.  If both $\alpha$ and $\beta$ have constant speed, then so does their
concatenation $\alpha\oplus\beta$.

\begin{proposition}\label{basics}
For any two curves $\alpha, \beta\in AC$,
\begin{enumerate}[(a)]
 \item \label{trans} $\mu_{\alpha+v}=\mu_\alpha$ for any $v\in\mathbb{R}^d$;
 \item \label{scale} $\mu_{c\alpha}=|c|\mu_\alpha$ for any scalar $c\in \mathbb{R}$;
 \item \label{reversal} $\mu_{\alpha^-}=\mu_\alpha$;
 \item \label{concat} $\mu_{\alpha\oplus\beta}=\mu_\alpha+\mu_\beta$;
 \item \label{reparam} If $\phi:[0,1]\rightarrow[0,1]$ is monotonic, absolutely continuous, and surjective, then 
 \[\mu_{\alpha\circ\phi}=\mu_\alpha.\]
\end{enumerate}
\end{proposition}
\begin{proof}
We prove these statements by proving analogous statements for the corresponding positive linear functionals $I_\alpha:C(P^{d-1})\rightarrow \R$ defined by
equation \eqref{functional}. \ref{trans} and \ref{scale} are clear.  \ref{reversal} is true by a change of variables and the fact that $\hat{f}$ is even for all 
$f\in C(P^{d-1})$.  The proof of \ref{reparam} is the same as the proof that $\mu_\alpha=\mu_{\alpha\circ\ti{\phi}}$ in theorem \ref{constspeedT}.\\

To prove \ref{concat}, note that the integral of $M_{\alpha\oplus\beta}$ can be written as the sum of two integrals over the intervals $[0,p]$ and $[p,1]$
 respectively.  The first integral is equal to $M_\alpha$ by a reparametrization.  The second is equal to $M_\beta$ by a translation and a reparametrization.
\end{proof}

\begin{definition}\label{ast}
Let $x \in L^\infty([0,1],\mathbb{R})$ and $\alpha\in AC$.  Then the function $x\alpha'$ belongs to $L^1[0,1]$, and it determines an absolutely continuous curve
\begin{align*}
\int_0^t x(s)\alpha'(s)\, ds.
\end{align*}
We let $x\ast\alpha$ denote the constant speed reparametrization of this curve.  
In the special case that $x$ is constant, then $x\ast\alpha=x\alpha$ (modulo translation).
\end{definition}

\begin{proposition}\label{recombination}
 Let $x(t)$ and $y(t)$ belong to $L^\infty([0,1],\mathbb{R})$, and assume $|x(t)|+|y(t)|=1$ almost everywhere. Then
 \[\mu_{x\ast\alpha\oplus y\ast\alpha}=\mu_\alpha.\]
\end{proposition}
\begin{proof}
 \begin{align*}
    I_{x\ast\alpha\oplus y\ast\alpha}(f) & =I_{x\ast\alpha}(f)+I_{y\ast\alpha}(f)\\
    & =\int_0^1 \hat{f}\left(x(t)\alpha'(t)\right)\, dt + \int_0^1 \hat{f}\left(y(t)\alpha'(t)\right)\, dt\\
    & =\int_0^1(|x(t)|+|y(t)|)\hat{f}(\alpha'(t))\, dt \\
    & = I_\alpha(f).\\
   \end{align*}
\end{proof}

\begin{example}\label{separation}
 Let $x(t)=p$, and $y(t)=q=1-p$, with both constants between $0$ and $1$.  Then $\mu_\alpha=\mu_{p\alpha\oplus q\alpha}.$
\end{example}

\begin{example}
Let $\{U_i:\, 1\leq i \leq n\}$ be pairwise disjoint measurable subsets of the unit interval $[0,1]$ such that $\bigcup_i U_i=[0,1]$.  And let $\chi_{U_i}$ be the indicator function on $U_i$:
\[ \chi_{U_i}(t)=\left\{
\begin{array}{ll} 1, & t\in U_i; \\ 0, & \text{otherwise}.\end{array}\right. .\]
Since $\sum_i |\chi_{U_i}(t)| = 1$, 
\begin{equation}\label{Rearrange}
\mu_{\bigoplus_i(\chi_{U_i}\ast\alpha)}=\mu_\alpha.
\end{equation}
\end{example}

\section{Construction of a right inverse from $\mM^+(P^{d-1})$ back to $AC$}
In this section, our goal is to show that every positive Borel measure on projective space arises as the measure $\mu_\alpha$ associated to an $AC$ curve $\alpha$.  We do this by approximating measures with discrete measures and 
$AC$ curves with broken lines.

\subsection{Partitions and discrete measures (notation)}
\begin{definition}
We call a measure $\mu\in \mM(P^{d-1})$ \emph{discrete} if it is a \emph{finite} linear combination of point masses.  The space of discrete
measures will be denoted $\mD(P^{d-1})$.
\end{definition}

\begin{definition}
 Let $X$ be a set endowed with a metric topology.
 We define a \emph{partition} $\mP$ of $X$ to be a finite collection of pairwise disjoint Borel sets $\mP=\{U_1,\, U_2,\,\dots,\, U_N\}$ whose union is $X$.  
 A \emph{tagged partition} $\mP^t$ is a collection of pairs $\mP^t=\{(U_i,x_i)\}_i$, such that:
 \begin{enumerate}[1.]
  \item The elements $U_i$ are Borel subsets of $X$ which form a partition $\mP=\{U_i\}_i$;
  \item For each $i$, $x_i$ is an element of $U_i$.
 \end{enumerate}
 The norm of a partition is defined to be the maximum of the diameters of the sets $U_i$:
 \begin{align*}
  \|\mP^t\|=\|\mP\|=\max_i \diam(U_i).
 \end{align*}
 A partition $\ti{\mP}$ is a \emph{refinement} of $\mP$ if each element of $\ti{\mP}$ is a subset of some element of $\mP$.  If the partitions are tagged, then we also require that the tags of $\mP^t$ be included in
the set of tags of $\ti{\mP^t}$.
\end{definition}

\begin{proposition}\label{dapprox}
 Let $\mu\in \mM(P^{d-1})$ be a finite Borel measure, and let $\mP^t=\{(U_i,x_i)\}_i$ be a tagged partition of $P^{d-1}$.    Then
 \begin{align*}
  \left\| \mu - \sum_i \mu(U_i)\delta_{x_i}\right\|_w < \|\mP^t\|\, |\mu|(P^{d-1}).
 \end{align*}
In particular, the set of discrete measures, $\mD(P^{d-1})$, is dense in the set of finite Borel measures $\mM(P^{d-1})$.
 \end{proposition}
 \begin{proof}
 Let $\mu_0$ denote the discrete measure being compared to $\mu$. If $f\in C(P^{d-1})$ is Lipschitz with $\|f\|_\mL \leq 1$, then
 \begin{align*}
  \left| \int f\, d(\mu-\mu_0)\right| & \leq \sum_i \left|\int_{U_i} f\, d(\mu-\mu_0)\right|\\
  & = \sum_i \left| \int_{U_i} \left(f(y)-f(x_i)\right)\, d\mu(y) \right|\\
  & \leq \sum_i \|f\|_{\mL}\diam(U_i)\,|\mu|(U_i) \\
  & \leq \|\mP^t\|\, |\mu|(P^{d-1}).
 \end{align*}
\end{proof}

\begin{definition}
 An \emph{ordered partition} $\mP=\{U_1,\, U_2,\, \dots\, U_n\}$ is a partition in which the sets $U_i$ are assigned a definite order.  
The ordered partition $\ti{\mP}=\{V_1,\, \dots,\, V_N\}$ is an \emph{ordered refinement} of $\mP$ if 
\begin{enumerate}[1.]
 \item the partition $\ti{\mP}$ is an ordinary refinement of $\mP$;
 \item If $i<j$ and $V_i\subset U_p$ and $V_j\subset U_q$, then $p\leq q$.
\end{enumerate}
Similarly, we define partitions that are tagged and ordered, and we define the relation of ordered refinement among these as well.
\end{definition}

\subsection{Constructing the right inverse}
Our first problem is that a directed line segment produces the same measure when the direction of the segment is reversed.  But given a measure, we would like 
to associate just one curve.  Therefore, for each element of projective space $x$, we would like to fix a representative element $x^+$ in $S^{d-1}$.  We do this by fixing a subset $X$ of the 
sphere which is in one to one correspondence with $P^{d-1}$ via the projection $\pi:P^{d-1}\rightarrow S^{d-1}$.\\

Let $\R^+$ be the set of positive real numbers, and let $y=(y_1,\, \dots,\, y_d)$ be coordinates 
in $\mathbb{R}^d$.  For each $i=1,\dots, d$, let $H_i$ be open half space $\{y_i > 0\}$.  Let $X$ denote the set
\begin{align}\label{X}
 X = S^{d-1} \cap \bigcup_{k=1}^d(\{y_i = 0, i> k\}\cap H_k).
\end{align}
Then $\pi$ maps $X$ bijectively onto $P^{d-1}$.  For all Borel sets $U\in P^{d-1}$, 
let $U^+$ denote $\pi^{-1}(U)\cap X$, and let $U^-=-U^+$ (which is disjoint from $U^+$).  Similarly, if $x\in P^{d-1}$,
we let $x^+=X\cap \pi^{-1}(x)$.\\

Let $\mP^t=\{(U_i,x_i):1\leq i\leq n\}$ be a tagged, ordered partition of $P^{d-1}$, and let $\alpha\in AC$.  For each $i=1,\dots ,n$, we define a function 
$h(\mP,\alpha)_i:[0,1]\rightarrow\R$ by the formula
\begin{align}\label{us}
h(\mP,\alpha)_i(t)=\left\{\begin{array}{ll} 1,& \alpha'(t)\in \R^+U^+_{i};\\
                           -1, & \alpha'(t)\in -\R^+U^+_{i};\\
                           0,& \text{otherwise}.
                          \end{array}\right.
 \end{align}
Through these functions, we define a mapping $\mF_\mP:AC\rightarrow AC$ given by
\begin{align*}
\mF_\mP (\alpha) = \bigoplus_i h(\mP,\alpha)_i\ast \alpha.
\end{align*}

Geometrically, this operation represents a sort of surgical rearrangement of the pieces of $\alpha$.  For each $i$, we cut out the parts of $\alpha$ that point
in directions parallel to the elements of $U_i$.  We then normalize these pieces so that they all point in directions of $U_i^+\subset X$, defined by \eqref{X}. Then we glue them together to
form $h(\mP,\alpha)_i\ast \alpha$.  Having done this for each $i$, we then concatenate these separate parts together.  The resulting curve 
still belongs to $AC$, it has constant speed, and it produces the same measure in $\mM(P^{d-1})$.  The advantage of the curve $\mF_\mP (\alpha)$ is 
that there is a partition of the interval
\begin{align*}
 0=t_0 \leq t_1 \leq \dots \leq t_n=1,
\end{align*}
such that for almost all $t\in (t_{i-1},t_i)$, $\alpha'(t)\in U^+_i$.  Moreover, 
\begin{align*}
(t_i-t_{i-1})\|\alpha\|_{AC} & = \|h(\mP,\alpha)_i\ast \alpha\|_{AC}\\
& = \mu_\alpha(U_i).
 \end{align*}

This suggests a method of finding a piecewise linear curve to approximate $\mF_{\mP}(\alpha)$.  
Let $\mP^t=\{(U_i,x_i)\}$ be ordered and tagged.  For each $x_i\in P^{d-1}$, we let $\hat{x_i}$ denote the linear segment in $AC$:
\begin{align*}
 \hat{x_i}:t\mapsto tx_i^+, 
\end{align*}
where $x_i^+$ is the unit vector representative of $x_i$ in the set $X$.\\

Define $\mG_{\mP^t}:AC\rightarrow AC$ by
\begin{align*}
 \mG_{\mP^t}(\alpha) & = \bigoplus_i \|h(\mP^t,\alpha)_i\ast\alpha\|_{AC}\, \hat{x_i}\\
  & = \bigoplus_i \mu_\alpha(U_i)\hat{x_i}.
 \end{align*}

Elements of the image of $\mG_{\mP^t}$ are broken lines.  Let $\ti{\mP}$ be an ordered refinement of $\mP$.  Then for almost all $t$, there exists a set $U_i$ in $\mP$ such that each of the vectors
 $\left(\mF_{\ti{\mP}}(\alpha)\right)'(t)$, $\left(\mF_{\mP}(\alpha)\right)'(t)$, and $\left(\mG_{\mP^t}(\alpha)\right)'(t)$ lies inside $\|\alpha\|_{AC}U_i^+$.  
 Consequently, we have the following inequalities:
 \begin{align}\label{part}
  \|\mF_{\mP^t}(\alpha)-\mG_{\mP^t}\alpha\|_{AC}\leq \|\mP^t\|\|\alpha\|_{AC};\\
  \|\mF_{\mP^t}(\alpha)-\mF_{\ti{\mP}^t}(\alpha)\|_{AC}\leq \|\mP^t\|\|\alpha\|_{AC};
 \end{align}
 
Given a tagged partition $\mP^t$, we define $G_{\mP^t}:\mM(P^{d-1})\rightarrow AC$ by
\begin{align*}
 G_{\mP^t}(\mu)=\bigoplus_i\mu(U_i)\hat{x_i}.
\end{align*}
The two definitions of $\mG_{\mP^t}$ match insofar as $G_{\mP^t}(\alpha)=G_{\mP^t}(\mu_\alpha)$.  

\begin{theorem}
 The correspondence $\alpha\mapsto \mu_\alpha$ is surjective onto $\mM^+(P^{d-1})$, the set of positive Borel measures.
\end{theorem}
\begin{proof}
Fix a positive measure $\mu$, and let $\mP^t_k$ be a sequence of ordered, tagged partitions of $P^{d-1}$ such that
 \begin{enumerate}[1.]
  \item For each $k$, $\mP^t_{k+1}$ is an ordered refinement of $\mP^t_k$;
  \item $\|\mP^t_k\|\rightarrow 0$ as $k\rightarrow \infty$.
 \end{enumerate}
Let $\alpha_k = \mG_{\mP^t_k}(\mu)$.  For all $k$, $\alpha_k$ has constant speed equal to $\mu(P^{d-1})$.  
If $k\leq l$, then for almost all $t$, there is a common $U\in \mP^t_k$ such that
$\alpha_k'(t)$ and $\alpha_l'(t)$ both lie inside $\mu(P^{d-1})U^+$.  Therefore, for almost all $t$,
\begin{align*}
 |\alpha_k'(t)-\alpha_l'(t)| & \leq \mu(P^{d-1})\diam(U) \\
& \leq \mu(P^{d-1})\|\mP^t_k\|.
\end{align*}
It follows that $\|\alpha_k-\alpha_l\|_{AC}\leq \|\mP^t_k\|\mu(P^{d-1})$.  Hence $\alpha_k$ is a Cauchy sequence and approaches a limit $\alpha$.
If we let $\mu_k = \mu_{\alpha_k}$, then by the continuity inequality \eqref{muCont},  $\mu_k\rightarrow \mu_\alpha$.  If we let $\{U_i\}$ denote the Borel sets 
in the partition $\mP_k$, then
\begin{align*}
 \mu_k=\sum_i \mu(U_i)\delta_{x_i}.
\end{align*}
By Proposition \ref{dapprox}, $\mu_k\rightarrow \mu$.  Hence $\mu_\alpha=\mu$.
\end{proof}

We note that the construction of $\alpha$ from $\mu$ is a right inverse of the map that sends $\alpha\mapsto \mu_\alpha$.

\appendix

\section{Absolutely continuous curves have constant speed reparametrizations}
The construction of the reparametrization is somewhat complicated because the speed of the given curve could be zero
on a set of positive measure.  The following theorem summarizes the important conclusions.\\

\begin{thmA}\label{constspeed}
Given $\alpha\in AC$, there exists a right-continuous, strictly monotonic function 
$\ti{\phi}:[0,1]\rightarrow [0,1]$ with the following properties:
\begin{enumerate}[i.]
\item \label{i} If $F:[0,1]\rightarrow \R^k$ is absolutely continuous and
\begin{align}\label{subord}
 \{t\, |\, F'(t) \neq 0\} \subset \{t\, |\, \alpha'(t)\neq 0\},
 \end{align}
 then $F\circ\ti{\phi}$ is absolutely continuous.  In particular $\alpha\circ\ti{\phi}$ is absolutely continuous.
\item \label{ii} The chain rule is satisfied almost everywhere:
\begin{align*}
 (\alpha\circ\ti{\phi})'(s)=\alpha'(\ti{\phi}(s))\ti{\phi}'(s).
\end{align*}
Moreover, $|(\alpha\circ\ti{\phi})'(s)|=\|\alpha\|_{AC}$ almost everywhere.
\item \label{iii} $\mu_{\alpha}=\mu_{\alpha\circ \ti{\phi}}$.  In particular $\|\alpha\|_{AC}=\|\alpha\circ\ti{\phi}\|_{AC}$.
\end{enumerate}
\end{thmA}

For this construction, Lebesgue measure is denoted by $\lambda$ or, in an integral, by $dx,\, dt,\, ds$ etc..  Characteristic functions, 
or indicator functions, shall be denoted by $\chi$.  For example, if $E\subset [0,1]$,
\[ \chi_E(t)=\left\{ \begin{array}{ll} 1, & t\in E;\\ 0, & t \notin E. \end{array}\right.\]

\begin{lemma*}
Let $E\subset [0,1]$ be a measurable set.  There exists a strictly monotonic, right-continuous function
$\phi:[0,1]\rightarrow [0,1]$ which, almost everywhere, satisfies
\[\chi_E(\phi(s))\phi'(s)=\lambda(E).\]
\end{lemma*}
\begin{proof}
If $\lambda(E)=0$, then $\phi(s)=s$ satisfies the conditions.  Assume $\lambda(E)>0$.\\

We define a function $\psi_E:[0,1]\rightarrow [0,1]$ by
\begin{equation}\label{psi}
\psi_E(t) = \frac{\lambda\left([0,t]\cap E\right)}{\lambda(E)}.
\end{equation}

The function $\psi_E$ is the cumulative distribution function of the measure $\frac{1}{\lambda(E)}\chi_E(t)dt$.  
Therefore, it is differentiable almost everywhere with derivative $\psi'=\frac{\chi_E}{\lambda(E)}$.
Also, it is monotonic and Lipschitz continuous with constant $\frac{1}{\lambda(E)}$.  
At the endpoints, we have $\psi_E(0)=0$ and $\psi_E(1)=1$. In particular $\psi_E$ is surjective.  
Hence, the following function $\phi_E:[0,1]\rightarrow[0,1]$ is well-defined:

\begin{equation}\label{phi}
\phi_E(s)= \sup \{t\, |\, \psi_E(t)=s\}.
\end{equation}

First, we prove that $\phi_E$ is strictly monotonic.  The function $\psi_E$ is continuous and monotonic.  Therefore, for all $s\in [0,1]$, 
$\psi_E^{-1}(s)$ is a closed interval $[a_s,\, b_s]$, with $b_s=\phi_E(s)$ by \eqref{phi}.
If $s_1 < s_2$, then $b_{s_1} < a_{s_2}$ by the monotonicity of $\psi_E$.  $b_{s_1}=\phi_E(s_1)$.  And $a_{s_2}\leq b_{s_2}=\phi_E(s_2)$.
Hence, 
\[\phi_E(s_1) <  \phi_E(s_1).\] 

For all $s$, $\psi_E(\phi_E(s))=s$.  Since $\psi_E$ is absolutely continuous and $\phi_E$ is monotonic, the chain rule
is valid almost everywhere (Corollary 3.50, cite{Leoni}).
\[ 1 = (\psi_E\circ\phi_E)'(s)= \psi_E'(\phi_E(s))\phi_E'(s)=\frac{\chi_E(\phi_E(s))}{\lambda(E)}\phi_E'(s).\]

Finally, we prove right-continuity.
$\phi_E(1)=1$, so for $s<1$, $\phi_E(s)<1$.  
Let $0\leq s < 1$.  We must prove that if $0 < \epsilon < 1 - \phi_E(s),$ then there exists $\delta>0$ such that
\[\phi_E(s+\delta)< \phi_E(s)+\epsilon.\]
\\
By \eqref{phi} and the monotonicity of $\psi_E$,
\[s= \psi_E(\phi_E(s)) < \psi_E(\phi_E(s)+\epsilon).\]
Define $t_1$ by the equation:
\begin{equation}\label{t1}
t_1=\inf \psi_E^{-1}\left(\psi_E(\phi_E(s)+\epsilon)\right)
\end{equation}
By \eqref{phi} and the monotonicity of $\psi_E$, $\phi_E(s)< t_1$.  Therefore, let $t_0$ be any point 
such that
\[ \phi_E(s) < t_0 < t_1.\]
Apply $\psi_E$ to obtain
\[ s=\psi_E(\phi_E(s)) < \psi_E(t_0) < \psi_E(t_1)=\psi_E(\phi_E(s)+\epsilon).\]
The inequalities are strict by \eqref{phi} and \eqref{t1} respectively.\\

Let $\delta=\psi_E(t_0)-s$.  Since $\phi_E$ is strictly increasing,
\[ \phi_E(s) < \phi_E\psi_E(t_0) < t_1 \leq \phi_E(s)+\epsilon.\]
\end{proof}

\begin{proof}[Proof of Theorem \ref{constspeed}]
Given a curve, $\alpha\in AC$, its velocity, $\alpha'$, could have two possible defects.

\begin{enumerate}[1.]
\item  It could be zero on a set of positive measure.
\item Its speed, $|\alpha'(t)|$, could vary with $t$.
\end{enumerate}

We correct these problems one at a time.
Let $E$ be the set on which $\alpha'\neq 0$.  
\[E=\{t\, |\, \alpha'(t)\neq 0\}.\]
Let $\psi=\psi_E$ and $\phi=\phi_E$ be the functions defined in the lemma.\\

Since the curve $\alpha$ is absolutely continuous and the function $\phi$ is monotonic, 
their composition $\alpha\circ\phi$ is differentiable almost everywhere and satisfies the chain rule 
(Corollary 3.50, \cite{Leoni}).  Let 
\[\ti{E}=\{s\, |\, (\alpha\circ\phi)'(s)=0\}.\]
We apply the chain rule and note that
\[ \alpha'(\phi(s))\phi'(s)=0\hspace{.2cm} \text{if and only if}\hspace{.2cm} \chi_E(\phi(s))\phi'(s)=0.\]

According to the lemma, the right hand side is $0$ at most on a set of measure $0$.\\

Next we prove that $\alpha\circ\phi\in AC$.  We start by describing the discontinuities of $\phi$.\\

The function $\phi$ has bounded variation.  Consequently it can have at most countably many discontinuities 
which must all be jump discontinuities.
As a right inverse of $\psi$, the jumps correspond to the intervals on which $\psi(t)$ is flat. 
These are, in other words, the intervals which have zero measure with 
respect to $\chi_E(t)\, dt$.  Denote these intervals $I_k\subset [0,1]$, and let $l_k$ denote their lengths $\lambda(I_k)$.  
We let $s_k=\psi(I_k)$ denote the corresponding point in the domain of $\phi$ at which the jump discontinuity occurs.  
By right-continuity,
\[\phi(s_k)-\phi(s_k^-)=l_k.\]
Also,
\[I_k=[\phi(s_k^-),\, \phi(s_k)].\]
By the fundamental theorem of calculus,
\[ \alpha(\phi(s_k))-\alpha(\phi(s_k^-))=\int_{I_k} \alpha'(t)\, dt.\]
The set $I_k\cap E$ has measure $0$, so the right hand integral equals $0$.  We conclude that $\alpha\circ\phi$ is continuous.\\

The curve $\alpha$ is absolutely continuous.  By defintion, if $\epsilon >0$,
there exists $\delta>0$ such that
\[ \sum_{i=1}^N |\alpha(d_i)-\alpha(c_i)| < \epsilon \]
for any finite collection of pairwise disjoint intervals $(c_i,\, d_i)\subset [0,1]$ satisfying
\[\sum_{i=1}^N (d_i-c_i) < \delta.\]

Now let $(a_i,\, b_i)\subset [0,1]$ be pairwise disjoint intervals such that
\[\sum_{i=1}^N (b_i-a_i) < \frac{1}{2}\delta.\]
We shall prove that
\[\sum_{i=1}^N |\alpha(\phi(b_i))-\alpha(\phi(a_i))|<\epsilon.\]

The sum of the jump discontinuities of $\phi$ cannot be greater than $1$: 
\[\sum l_k < 1.\] It follows that for sufficiently large $M$,
\[\sum_{k\geq M} l_k < \frac{1}{2}\delta.\]

If $1\leq k\leq M-1$ and the discontinuity $s_k$ lies inside the interval $(a_i,b_i)$, we subdivide the 
interval into two new intervals $(a_i,s_k)$ and $(s_k,b_i)$.  Repeating this process at most $M-1$ times, we obtain a new
collection of disjoint intervals $(A_i,B_i)$ in which the jump discontinuities $s_k: 1\leq k\leq M-1$ 
can only occur at the endpoints.  We let $\phi(B_i^-)$ and $\phi(A_i^+)$ respectively denote the left hand
and right hand limits.  By right-continuity, $\phi(A_i^+)=\phi(A_i)$.\\

By the lemma, $\|\phi'\|_{\infty}=\lambda(E)\leq 1$.  It follows that 
\[ \sum_i \phi(B_i^-)-\phi(A_i) \leq \sum_i (b_i-a_i)+\sum_{k\geq M}l_{k}<\delta.\]

Since $\phi$ is strictly monotonic, the intervals $(\phi(A_i),\phi(B_i^-))$ are pairwise disjoint.  
By the absolute continuity of $\alpha$, we conclude that 
\[\sum_{i=1} |\alpha(\phi(B_i^-))-\alpha(\phi(A_i))|<\epsilon.\]
Since $\alpha\circ\phi$ is continuous, $\alpha(\phi(B_i^-))=\alpha(\phi(B_i)).$  By the triangle inequality, it follows that
\[\sum_{i=1} |\alpha(\phi(b_i))-\alpha(\phi(a_i))|<\epsilon.\]

The arguments above remain valid if we replace $\alpha$ with any absolutely continuous function $F$ that satisfies \eqref{subord}.  
Hence, for any such $F$, $F\circ\ti{\phi}$
is absolutely continuous.  This applies, in particular, to the function 
\[ \ell(t) = \frac{1}{\|\alpha\|_{AC}}\int_0^{t} |\alpha'(\tau)|\, d\tau.\]

We have adjusted the parametrization of the curve $\alpha$ so that its speed is $0$ on, at most, a set of measure $0$.  
To finish the proof, we must show that we can further reparametrize to normalize the speed.\\

The function $\ell\circ\phi$ is absolutely continuous, monotonic, and $\{s\, |\, (\ell\circ\phi)'(s)=0\}$ has measure $0$.  
Therefore the inverse function
\[ (\ell\circ\phi)^{-1}:[0,1]\rightarrow [0,1]\]
is also absolutely continuous and monotonic (\cite{Leoni}).  Define
\[\ti{\phi}=\phi\circ(\ell\circ\phi)^{-1}.\]
Both $\alpha\circ\phi$ and $(\ell\circ\phi)^{-1}$ are absolutely continuous and the latter is monotonic.  Therefore, their composition
\[(\alpha\circ\phi)\circ(\ell\circ\phi)^{-1}=\alpha\circ\ti{\phi}\]
is absolutely continuous.  Similarly, $F\circ\ti{\phi}$ , and $\ell\circ\ti{\phi}$ are absolutely continuous.
Thus we have proved the first conclusion, \ref{i}.\\

If $s=(\ell\circ\phi)^{-1}(q),$ the chain rule implies
\[ \frac{d}{dq}(\ell\circ\phi)^{-1}(q) = \frac{\|\alpha\|_{AC}}{|(\alpha\circ\phi)'(s)|}.\]
Therefore, 
\[ |(\alpha\circ\ti{\phi})'(q)| = |(\alpha\circ\phi)'(s)\frac{d}{dq}(\ell\circ\phi)^{-1}(q)| = \|\alpha\|_{AC}.\]
This proves the second statement, \ref{ii}.
 
To prove the final statement, let $f$ be a continuos function on $P^{d-1}$.  Define $F(t)$ by
\begin{align*}
 F(t)=\int_0^t \hat{f}(\alpha'(q))\, dq.
\end{align*}
Then the inclusion \eqref{subord} is satisfied which implies that $F(\ti{\phi}(s))$ is absolutely continuous.  
By Theorem 3.54 (\cite{Leoni}), this justifies the following change of variables:
\begin{align}
 \int_0^{\ti{\phi}(s)} \hat{f}(\alpha'(q))\, dq & = \int_0^s \hat{f}(\alpha'(\ti{\phi}(s))\ti{\phi}'(s)\, ds\\
 & =\int_0^s \hat{f}\left((\alpha\circ\ti{\phi})'(s)\right)\, ds.
\end{align}
Setting $s=1$, we conclude $\int f\, d\mu_\alpha  =\int f\, d\mu_{\alpha\circ \ti{\phi}}$ for all $f\in C(P^{d-1})$.  
Letting $f=1$, we find that $\|\alpha\|_{AC}=\|\alpha\circ\ti{\phi}\|_{AC}$.
\end{proof}

\begin{remark}
 The reader might wonder why the construction of $\ti{\phi}$ proceeded in two steps.  Indeed, one could 
 directly define $\ti{\phi}$ as the right inverse of $\ell$:
 \[\ti{\phi}=\sup\{t\, |\, \ell(t)=s\}.\]
Since $\ell$ and $\psi$ have many of the same properties, $\ti{\phi}$ shares many key properties with $\phi$.  
The problem with this direct approach is in
the proof of the absolute continuity of $\alpha\circ\phi$.  We used the fact that $\|\phi'\|_{\infty}\leq 1$.  In contrast
$\ti{\phi}'$ need not belong to $L^\infty$.
\end{remark}

\bibliography{geotomo}
\bibliographystyle{amsplain}

\end{document}